\documentclass[graybox]{svmult}

\usepackage{mathptmx}       
\usepackage{helvet}         
\usepackage{courier}        
\usepackage{type1cm}        
\usepackage{makeidx}         
\usepackage{graphicx}        
\usepackage{multicol}        
\usepackage[bottom]{footmisc}

\usepackage{amsmath}
\usepackage{amsfonts}
\usepackage{epsfig}
\usepackage{bm}

\newcommand{\las}{\lambda_{\,\mathrm{s}}}
\newcommand{\laf}{\lambda_{\,\mathrm{f}}}

\newcommand{\pa}{\partial}
\newcommand{\R}{\mathrm{I\! R}}

\newcommand{\f}{\mathfrak{f}}

\makeindex             

\begin{document}

\title*{Explicit construction of effective flux functions for Riemann solutions}
\author{Pablo Casta\~neda}
\institute{Pablo Casta\~neda
\at Department of Mathematics, ITAM \\
R\a'io Hondo 1, Ciudad de M\'exico 01080, Mexico. \\
\email{pablo.castaneda@itam.mx}
}

\maketitle

\abstract{For a family of Riemann problems for systems of conservation laws, we construct a flux function that is scalar and is capable of describing the Riemann solution of the original system.}

\section{Introduction}

\noindent
We are interested in injection problems leading to flow in porous media, which are modeled by systems of conservation laws; a survey of the mathematical theory for such flow may be found in \cite{Aze14,Hans,Schecter} and references therein. In this work we focus on the Riemann problems and their solutions via the Wave Curve Method ({\it cf.} \cite{Aze10,CAFM16}) and on the construction of effective flux functions (EFF) allowing to understand the whole system as a single scalar conservation law, see \cite{CAFM16,Oleinik57}.

The setting for such a construction is given for a fixed state in physical space. Thus we can develop the construction of a wave group the {\it lifting} of which will give an effective flux function. Such functions can be treated as the flux functions of a scalar conservation law in a certain parametrized coordinate. This lifting is the crucial part in the construction. However, we will show that this function is not unique. We only have uniqueness as a class of functions, the representation of which will be the effective flux function for each state, each starting eigenvalue family and, the chosen coordinate system.

Analogous effective flux functions have been used satisfactorily in many works. There are implicit uses in \cite{Aze10,Aze14,BC16,RM12} and, explicit constructions in \cite{CAFM16,preprint}. Another potential application is modeling special flux functions in experimental data for which classical known models are inappropriate ({\it e.g.} \cite{GetA15} {\it vs.} \cite{RAA01}).

\section{The 2$\times$2 system of conservation laws}\label{sec:system}
Let us write as a system the conservation laws used along this work. We focus in a system with two equations, since it is only important to take more than one equation. The extension to any number of equations would be natural.

Let $u_1(x, t)$, $u_2(x, t)$ to be the conserved quantities at distance $x$ along the real axis, at time $t$. Typically in one spatial dimension a set of equations governing the system
\begin{equation}
\frac{\partial U}{\partial t} + \frac{\partial F}{\partial x} = 0,
\quad \mbox{ or } \quad
\begin{array}{r}
{\partial u_1}/{\partial t} + {\partial f_1}/{\partial x} = 0 \\
{\partial u_2}/{\partial t} + {\partial f_2}/{\partial x} = 0,
\end{array}
\label{eq:conservation}
\end{equation}
for $x \in \R$, $t \geq 0$, representing the conservation of $U = (u_1,\,u_2)$. The flow functions characterize the system and are denoted as the vector $F(U) = (f_1(u_1,\,u_2),\,f_2(u_1,\,u_2))$.
We denote as $\mathcal{D}$ the space of states $U$, in general we consider $\mathcal{D} \subset \R^2$.

Of special interest in applications and numerical calculations consists in the classification of the solution structure of the system of PDE \eqref{eq:conservation} with discontinuous data:
\begin{align} \label{eq:Riemanndata}
    U(x,\,t=0) = \begin{cases}
        U^L \quad \text{if} \quad x < 0, \\
        U^R \quad \text{if} \quad x > 0 .
    \end{cases}
\end{align}
The Riemann problem consists of system \eqref{eq:conservation} with initial Riemann data \eqref{eq:Riemanndata}, which we will denote as $\mathrm{RP}(U^L,\,U^R)$, with left and right values $U^L$ and $U^R$, respectively.

Strictly hyperbolic systems of conservation laws, {\it i.e.}, where the eigenvalues of the Jacobian matrix of the flux function are real and distinct, provide a relatively well understood framework for the solution of Riemann problems~\cite{Daf00}. Here we assume that the system is not necessarily strictly hyperbolic.

A remarkable observation is that in many problems for flow in porous media, there is an extra conserved quantity or fluid, thus an extra equation for system \eqref{eq:conservation}. Typically the extra quantity and the phases in \eqref{eq:conservation} add up to one. In the same way, there is an extra flux function that depends on these phases, see for example \cite{Hans,CAFM16,RM12} where $u_1$, $u_2$, and $u_3$ can represent the saturation of water, gas and oil. Here the state space is the saturation triangle defined by $U$ satisfying $0 \leq u_1,\,u_2$, $u_1 + u_2 \leq 1$ and the constraint $u_1 + u_2 + u_3 = 1$. For this model an extra function satisfies $f_3 = 1 - f_1 - f_2$ that gives a third redundant equation. This is an important fact since the parametrization of the EFF can be given in any of those coordinates.

\subsection{Terse review of Fundamental Waves}
Equations~\eqref{eq:conservation} have solutions that propagate as nonlinear waves.
Because of {\it self-similarity} of the data and the PDE, the solutions of a Riemann problem depend on $x/t$
and consist of centered rarefaction waves, shock waves and sectors of constant states,
see {\it e.g.} \cite{Lax57,Liu74,Oleinik57}.
The characteristic speeds are the two eigenvalues of the Jacobian matrix
\[ \mathrm{J}(S) \;:=\; \frac{\pa (f_1(U),\,f_2(U))}{\pa (u_1,\,u_2)} \;=\; \frac{\pa F(U)}{\pa U}. \]
When the eigenvalues are distinct and real we say that the system is strictly hyperbolic. Sometimes, it loses hyperbolicity at particular states, as the cases registered in \cite{BC--,MCM12,RM12} for umbilic and quasi-umbilic points.
For distinct eigenvalues, the smaller and larger are called the slow- and the fast-family characteristic speed.
Actually, these eigenvalues can be equal on a curves or on larger sets, see \cite{SM14}.

System~\eqref{eq:conservation} has smooth solutions called (slow- and fast-family) rarefaction waves.
They arise by solving an ODE, namely,
\begin{equation}\label{eq:rar}
{dU}/{d\xi} = \vec{r}_k(U),
\end{equation}
$k = s$ or $f$ ({slow or fast}), defined by the eigenvectors of the Jacobian matrix $\mathrm{J}(U)$, with
\begin{equation}\label{eq:eigenproblem}
\{\mathrm{J}(U) - \xi\,\mathrm{I}\} \vec{r}_k(U) = 0,
\end{equation}
where $U(\xi)$, for $\xi = x/t$, is the profile of the {\it forward} rarefaction, provided $\xi$ is monotone increasing; it is called {\it backward} for $\xi$ monotone decreasing. The $k$-integral curve of a state $U^o$, denoted by $\mathcal{R}_k(U^o)$, consists of all states $U$ for each of which $U(\xi)$ solves the initial value problem \eqref{eq:rar} for the given $k$ and initial condition $U^o$, either backward and forward.

This system also admits solutions in the form of moving jump discontinuities. In order to respect conservation of $u_1$ and $u_2$, the fluxes in and out of the moving discontinuity must balance.
In terms of the state $U^o = (u^o_1,\,u^o_2)$ on the left of the discontinuity, the state $U = (u_1,\,u_2)$
on the right of the discontinuity, and the propagation speed $\sigma$, this balance is expressed as $F(U) - \sigma U =  F(U^o) - \sigma U^o$, or
\begin{equation}
\label{rel:RH}
\begin{array}{rcl}
f_1(U) - \sigma u_1 &=& f_1(U^o) - \sigma u^o_1, \\
f_2(U) - \sigma u_2 &=& f_2(U^o) - \sigma u^o_2.
\end{array}
\end{equation}
Eqs.~\eqref{rel:RH} are the Rankine-Hugoniot (RH) conditions. The Rankine-Hugoniot locus of a state $U^o$, denoted by $\mathcal{H}(U^o)$, consists of all states $U$ for each of which
there exists a value $\sigma = \sigma(U^o,\,U)$ such that the RH conditions~\eqref{rel:RH} are satisfied.

The two former loci in the space of states will be the basis for the construction of the wave group and for the effective flux function. We recall that in practice we must select a criterion for the ``physically admissible'' discontinuities that appear in the solutions. This will be Liu's criterion including also the Lax's admissibility criteria, namely:

\smallskip
\smallskip
\noindent
{\it Shock admissibility.} A discontinuity with propagation speed $\sigma = \sigma(U^L,\,U^R)$ between a left state $U^L$ and a right state $U^R$ is admissible if it satisfies Liu's admissibility criterion \cite{Liu75}, whenever it is applicable. We use Lax's admissibility criterion \cite{Lax57} in order to classify the propagation speed as follows:
\begin{equation}
\label{rel:Lax}
\begin{array}{rl}
\las(U^R) < \sigma < \las(U^L), \;&\; \sigma < \laf(U^R), \quad
\mbox{for slow-family shocks}, \\
\laf(U^R) < \sigma < \laf(U^L), \;&\; \las(U^L) < \sigma, \quad 
\mbox{for fast-family shocks}.
\end{array}
\end{equation}

\smallskip
\noindent
Moreover, in the previous definitions we allow one of the inequalities to become an equality;
hence our admissibility criterion. The nomenclature slow- and fast-family originates from the 1-Lax and 2-Lax shock waves, \cite{Lax57}.

The following definition is inspired by the Welge-Ole\u{\i}nik's construction for a single conservation equation. Let $U^o$ be a state in physical space;
a state $U$ is a {\it slow} (respectively {\it fast}) {\it extension} of $U^o$ if $U$ belongs to $\mathcal{H}(U^o)$ and the shock speed $\sigma(U^o,\,U)$ equals the
slow characteristic speed $\las(U)$ (resp. fast $\laf(U)$), {\it i.e.}, the shock is characteristic at $U$.

The Bethe-Wendroff Theorem ({\it cf.} \cite{Fur89}) guarantees that at an extension point one of the rarefactions curves starting at $U$ is tangent to the $\mathcal{H}(U^o)$ at $U$.

\subsection{The Wave Curve Method}
\label{sec:MoC}

Solutions found by Wave Curve Method consist of rarefaction fans, shock discontinuities and constant states, a survey is found in \cite{Aze10}. The classical construction is guaranteed to succeed only when $U^L$ and $U^R$ are close.
The Buckley-Leverett (BL) solution exhibits inflection points ({\it cf.} \cite{BL42, Oleinik57}), where equalities such as $\nabla\lambda_k(U)\cdot r_k(U) = 0$ for $k = s$ or $f$ occur. The BL shows that rarefaction waves and shock waves of the same family can be adjacent. For adjacency to occur, the shock speed must coincide with the same-family characteristic speed at an edge of the rarefaction wave, and one of the Lax's inequalities in \eqref{rel:Lax} becomes an equality.

Of course, when traversing the solution by increasing $x/t$ monotonically, the corresponding wave speed must also increase. Consequently, fast-family waves follow slow-family waves. This structural feature was first identified by Liu, \cite{Liu74}, under technical restrictions and holds in general; see \cite{Schecter}. Wave sequences are concatenated following certain rules.

\begin{figure}[htb]
\resizebox{5.32cm}{!}{\includegraphics{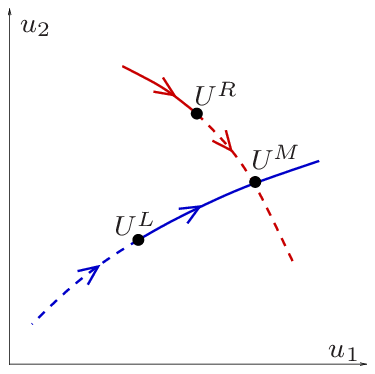}}
\hspace{0.9cm}
\resizebox{5.32cm}{!}{\includegraphics{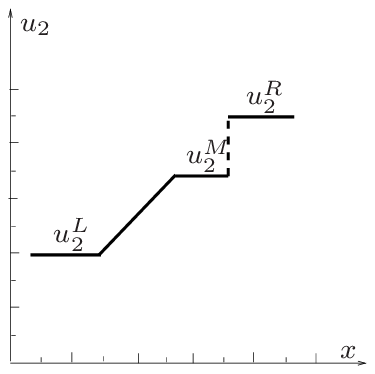}}
\caption{Solution of RP$(U^L,\,U^R)$, $U^L$ determines the slow-family wave curve and $U^R$ the backward fast-family wave curve. Their intersection determines $U^M$. At the right figure, the solution profile at a fixed time, notice the value of both coordinate system at $u_2$.}
\label{fig:wcm}
\end{figure}

A wave group is a sequence of waves, all associated to the same family, which are adjacent, meaning that no two waves are separated by a constant state. As in the Lax construction (Fig.~\ref{fig:wcm}), if the beginning or end state of a wave group is prescribed, then the state on the opposite end of the wave group lies on a curve in state space called a wave curve. In the language of wave groups, the Riemann solution consists of the following wave sequence, from left to right: the left state $U^L$, a slow-family wave group, an intermediate state $U^M$, a fast-family wave group, and the right state $U^R$. Notice that the slow-family wave curve from $U^L$ consists of all states $U^M$ attainable through a slow-family wave group; similarly, the backward fast-family wave curve from $U^R$ consists of all states $U^M$ attainable through a fast-family wave group. 
Wave curves and an algorithm to construct the Riemann solution were described in \cite{Liu74}.

In constructing wave groups with two or more waves, in \cite{Liu74} Liu introduced a shock wave
admissibility criterion that encompasses the criterion of Ole\u{\i}nik, \cite{Oleinik57}. The latter is a generalization of Welge's construction, the one we use for constructing the wave group,
hence we call Welge point to values where a shock wave is characteristic for a scalar flux function.
The wave curve method is applicable to systems with any number of conservation laws.
The effective flux function is given for each wave group as we show soon.

\section{The effective flux function construction}
The main idea is the following: for any given fixed state $R$ in $\mathcal{D}$ construct a {\it base curve} $\Gamma: I \rightarrow \mathcal{D}$ and an {\it Effective Flux Function} (EFF) $\f:I \rightarrow \R$. The base curve $\Gamma$ is a parametrization of a wave group; for the $2 \times 2$ systems we have at least four ways to starting such a curve.
The EFF $\f(\ell)$ is the {\it lifting} of the base curve defined by the wave group on the physical space.

The heart of our work resides in two facts: (1) when the state $\Gamma(\ell)$ is at a shock curve, the RH condition \eqref{rel:RH} is satisfied, thus the shock speed can be given by any of those relations or a linear combination of them, (2) when the state $\Gamma(\ell)$ is at an integral curve, the wave speed is given as $\lambda_k\big(\Gamma(\ell)\big)$ with $k = s$ or $f$, the corresponding family of such a rarefaction.

Now let us explain the construction of the base curve and how its lifting is found. As pointed out, we subdivide these constructions into two cases. This is done only for an easier exposition because it always works in the same manner.

\subsection{Construction of the base curve $\boldsymbol{\Gamma(\ell)}$}
\label{sec:base}

In Sec.~\ref{sec:MoC} it is explained how to construct slow- and fast-family wave groups with integral curves and Hugoniot loci on the state space. For a hyperbolic point, {\it i.e.} where both characteristic speeds are real and distinct, there are two linearly independent eigenvectors that describe the tangents to the slow and fast wave groups.

Typically, for each family the respective eigenvalue increases in the direction of an eigenvector and decreases in the opposite direction. Therefore, for a forward wave group construction, we have a rarefaction wave in one direction and a shock wave in the opposite one. There are cases where both directions have shock or rarefactions as initial waves, such points are in the inflection manifold of the respective family.

The base curve $\Gamma: I \rightarrow \mathcal{D}$ will be a parametrization of one wave group. The choice of the interval of interest $I$ and its parametrization is essential in the construction of the EFF $\f(\ell)$, its {\it lifting}, the construction of which will be given soon.

First of all, we need some guiding notation for the analysis that follows. As $\Gamma(\ell)$ belongs to $\mathcal{D}$ we take the coordinates as $\Gamma(\ell) = (\gamma_1(\ell),\,\gamma_2(\ell))$. Thus for $U \in \Gamma(\ell)$ we have $u_i = \gamma_i(\ell)$ for $i = 1,\,2$. Actually, if one of the coordinates $\gamma_1$ or $\gamma_2$ is monotonic, say $\gamma_1$, then it is possible to do a reparametrization $\ell = \gamma_1(\ell)$. Thus we can assume that $\Gamma(\ell) = (\ell,\,\gamma_2(\ell))$ holds at least locally. (In \cite{CAFM16} the parametrization is given with $\ell$ as the oil saturation, the third implicit coordinate, and it is easy to see that it is always possible to take $\ell$ as a linear combination of the coordinates $\gamma_1$ and $\gamma_2$; we assume that $\gamma_1(\ell) = \ell$ holds for an easier exposition.)

In the following sections the {lifting} construction of the EFF $\f(\ell)$ is described. Nonetheless, it is important to remark that such a lifting has as motivation to be a function that behaves as a single scalar flux function. Therefore, the Ole\u{\i}nik E-criterion is applicable ({\it cf.} \cite{Oleinik57}), moreover as such a construction is based on the envelope of the flux function, the ``mirror effect'' given in \cite{Cas16} follows.

\subsection{EFF construction: the first wave is a shock wave}
\label{sec:shockEFF}

Assume that the base curve $\Gamma(\ell)$ starts with a $k$-Lax shock wave at the reference state $R = (u^R_1,\,u^R_2)$. As pointed out in Sec.~\ref{sec:MoC} and in Fig.~\ref{fig:wcm}, the forward wave group is relevant for $k = s$ and the backward wave group for $k = f$. In the forward construction we take the part of $\mathcal{H}(R)$ that satisfies \eqref{rel:Lax} for the chosen $k$ and such that Liu's criterion holds for all points between $R$ and $U$, respectively as left and right states, see \cite{Liu74,Liu75}. (For backward construction, $U$ and $R$ are the left and right states.) Thus for each $U = (u_1,\,u_2)$ in $\mathcal{H}(R)$ there exists $\gamma_2(\ell)$ such that $(\ell,\,\gamma_2(\ell)) = U$ holds; see Sec.~\ref{sec:base}. Notice that the set of admissible shocks may stop at a Bethe-Wendroff point $U^*$ where $\sigma(R,\,U^*) = \lambda_{k'}(U^*)$ occurs. (Actually $k' \neq k$ may hold.)

We set the interval $I = [u^R_1,\,u^*_1]$ assuming that $u^R_1 < u^*_1$. (Conversely, $I = [u^*_1,\,u^R_1]$ for $u^*_1 < u^R_1$.) Now we have the base curve $\Gamma: I \rightarrow \mathcal{D}$ satisfying $\Gamma(\ell) = (\ell,\,\gamma_2(\ell)) \in \mathcal{H}(R)$ the position of which determines the corresponding shock speed \eqref{rel:RH}. Thus, we construct the {\it lifting} as
\begin{equation}
\label{EFF:shock}
\begin{array}{rcccl}
\f &\; : \;& I    & \;\longrightarrow\; & \R \\
   &       & \ell & \;\longmapsto    \; & \f(\ell) := f_1(R) + \sigma\big(R,\,\Gamma(\ell)\big)[\ell - u_1^R]
\end{array},
\end{equation}
thus, notice that the scalar shock speed satisfies
\begin{equation}\label{eq:restrRH}
\sigma(u^R_1,\,\ell) \;=\; \frac{\f(\ell) - f_1(R)}{\ell - u_1^R} \;=\; \sigma\big(R,\,\Gamma(\ell)\big),\qquad \mbox{for all}\qquad \ell \in I,
\end{equation}
so the Rankine-Hugoniot condition \eqref{rel:RH} holds as desired.

\begin{remark}
When $\ell^*$ is a value of a Welge point of the EFF (always with left state $u^R_1$), then $U^* = \Gamma(\ell^*)$ is a Bethe-Wendroff point of $\mathcal{H}(R)$, thus the eigenvector points parallel to the Hugoniot locus and the wave curve may be followed by a rarefaction curve. (As stated by the Bethe-Wendroff theorem.)
\end{remark}

\begin{lemma}\label{rem:high}
The EFF $\f(\ell)$ in \eqref{EFF:shock} is the first flux function over the base curve, \emph{i.e.}, $\f(\ell) = f_1\big(\Gamma(\ell)\big)$.
\end{lemma}

\begin{proof}
The RHS of \eqref{eq:restrRH} is equivalent to $[f_1\big(\Gamma(\ell)\big) - f_1(R)]/[\ell - u_1^R]$, equating with the middle term in \eqref{eq:restrRH} proves that $\f(\ell) = f_1\big(\Gamma(\ell)\big)$ holds in \eqref{EFF:shock}. \hfill $\square$
\end{proof}

The Bethe-Wendroff point at the state $U^*$ can be taken as a new starting reference point for a rarefaction wave. Just make sure that the starting family $k$ now must be taken as $k'$; they may not be the same.

\subsection{EFF construction: the first wave is a rarefaction fan}
\label{sec:rarEFF}

Assume that the base curve $\Gamma(\ell)$ starts with a $k$-rarefaction curve at the reference state $R = (u^R_1,\,u^R_2)$.
In the forward construction we take the $\mathcal{R}_k(R)$ part which has increasing eigenvalue $\xi = \lambda_k(U)$ for $U \in \mathcal{R}_k(R)$. (In the backward construction we take the decreasing eigenvalue direction.) Thus for each $U = (u_1,\,u_2)$ in $\mathcal{R}_k(R)$ there exists $\gamma_2(\ell)$ such that $(\ell,\,\gamma_2(\ell)) = U$ holds; see Sec.~\ref{sec:base}.

Recall that the rarefaction curve may stop at an inflection point $U^*$ (where $\nabla\lambda_k(U^*) \cdot r_k = 0$ occurs), then we set the interval $I = [u^R_1,\,u^*_1]$ assuming that $u^R_1 < u^*_1$. (Conversely, $I = [u^*_1,\,u^R_1]$ for $u^*_1 < u^R_1$.) Thus, the base curve $\Gamma: I \rightarrow \mathcal{D}$ satisfies $\Gamma(\ell) = (\ell,\,\gamma_2(\ell)) \in \mathcal{R}_k(R)$, the position of which determines the corresponding eigenvalue. Thus, we construct the {\it lifting} as
\begin{equation}
\label{EFF:rar}
\begin{array}{rcccl}
\f &\; : \;& I    & \;\longrightarrow\; & \R \\
   &       & \ell & \;\longmapsto    \; & \f(\ell) := f_1(R) + \int_{u_1^R}^{\ell}\,\lambda_k\big(\Gamma(t)\big)\,dt
\end{array},
\end{equation}
and notice that $\f'(\ell) = \lambda_k(\Gamma(\ell))$ holds as desired.

\begin{lemma}\label{rem:f_inRar}
The EFF $\f(\ell)$ in \eqref{EFF:rar} is the first flux function over the base curve, \emph{i.e.}, $\f(\ell) = f_1\big(\Gamma(\ell)\big)$.
\end{lemma}

\begin{proof}
Direct differentiation in \eqref{EFF:rar} shows that $\f'(\ell) = \lambda_k(\Gamma(\ell))$ holds, notice also that $\nabla \Gamma(\ell) = (1,\,\gamma_2'(\ell))$ is parallel to $\vec{r}_k(\Gamma(\ell))$. Thus, the identity $\mathrm{J}(\Gamma(\ell)) \nabla \Gamma(\ell) = \lambda_k(\Gamma(\ell))\,\nabla \Gamma(\ell)$ holds, see \eqref{eq:eigenproblem}, the first coordinate indicates that
$$\frac{d}{d\ell}f_1(\Gamma(\ell)) \;=\; \bigg(\frac{\pa f_1}{\pa u_1},\,\frac{\pa f_1}{\pa u_2}\bigg) \cdot (1,\,\gamma_2'(\ell)) \;=\; \lambda_k(\Gamma(\ell))$$
is satisfied. Then, as $\f(\ell)$ and $f_1(\Gamma(\ell))$ solve the same IVP with initial condition $f_1(R)$, the fluxes are the same. \hfill $\square$
\end{proof}

The inflection at the state $U^*$ can be taken as a new starting reference point for a shock wave, just make sure that in \eqref{EFF:shock} the reference state for determining the Hugoniot locus and the shock speed is the original reference state $R$ and not $U^*$.

\subsection{The complete EFF construction}

In previous sections we depicted the construction of a {base curve} $\Gamma(\ell)$, see Sec.~\ref{sec:base}, and two ways for the {lifting} of $\f(\ell)$ based on shock curves (Sec.~\ref{sec:shockEFF}) or on rarefaction curves (Sec.~\ref{sec:rarEFF}). In Sec.~\ref{sec:MoC} we pointed out that a wave curve is a composition of the former waves; it changes types at inflection or Bethe-Wendroff points. In this section we show the construction of a complete EFF which is actually a smooth function.

\smallskip
\smallskip
\noindent
{\it From shock to rarefaction curve.} Once we start a {base curve} with shock waves, this is a curve along $\mathcal{H}(R)$ and  would change to an integral curve within the same wave only at a Bethe-Wendroff point $U^*$. Notice that at such a point the equality $\f(u_1^*) = \lambda_k(U^*)$ holds for certain $k$, therefore the continuity for the derivatives of the {lifting} between both expressions \eqref{EFF:shock} and \eqref{EFF:rar} holds. 

The continuity of the EFF itself holds because of the adding of $f_1(R)$ in both {liftings} \eqref{EFF:shock} and \eqref{EFF:rar}: from Lemma \ref{rem:high} we have that $\f(\ell) = f_1\big(\Gamma(\ell)\big)$ holds in particular at $U^*$, and since from \eqref{EFF:rar} we have that $\f(u_1^*) = f_1(U^*)$ holds, also the EFF continuity.

\smallskip
\smallskip
\noindent
{\it From rarefaction to shock curve.} The continuity at this transition do not seems so natural; the actual value of $\f(\ell)$ in \eqref{EFF:rar} is not know {\it a priori} for a given $\ell$. However as the transition must occur at a point $U^*$ belonging to both $\mathcal{R}_k(R)$ and $\mathcal{H}(R)$, thus the values $\lambda_k(U^*)$ and $\sigma(R,\,U^*)$ agree and from the latter and \eqref{EFF:shock}, we notice that $\f(u_1^*) = f_1(U^*)$ holds; hence the smoothness at such transitions.

\smallskip
\smallskip
For a complete construction of an EFF, we notice that transitions from shock waves to rarefaction waves and vice-versa may occur many times. This is perfectly controlled by our way of doing the {liftings}; recall that from a shock to a rarefaction, the family $k$ may have changed to $k'$ and that from a rarefaction to a shock, the reference point $R$ is the original one. We have proven our main result:

\begin{theorem}
Let $R$ to be a fixed state in $\mathcal{D}$. Construct a wave curve through $R$. Select a coordinate $\ell$, let us say $\ell = u_1$, to parametrize the curve as $\Gamma:I \rightarrow \mathcal{D}$ satisfying $\Gamma(u_i^R) = R$. Therefore, an EFF is the respective flux function along $\Gamma(\ell)$ as
$$\begin{array}{rcccl}
\f &\; : \;& I    & \;\longrightarrow\; & \R \\
   &       & \ell & \;\longmapsto    \; & f_1\big(\Gamma(\ell)\big)
\end{array}.$$
\end{theorem}

\begin{proof}
See Lemmas \ref{rem:high} and \ref{rem:f_inRar}. \hfill $\square$
\end{proof}

Just recall that such an EFF is constructed based on $R$, thus even for $R' \in \Gamma(\ell)$, the respective EFF may be distinct. In the next section we show some examples.

\section{Some examples and applicability}
In this section we explain two examples. The first one based in the satisfactory use of EFFs in \cite{CAFM16,preprint}, constructed along the so-called separatrix as a crucial wave group for the Riemann solution. The second one shows how choosing the parametrization coordinate is important for understanding an EFF.

\subsection{Simplified quadratic Corey model}
In Enhanced Oil Recovery (EOR) the proposed conservation laws are given by fractional flux functions due to physical features as rock and fluid permeabilities (see {\it e.g.} \cite{Hans,RM12}). Here we take a schematic model with flow functions for \eqref{eq:conservation} given by:
\begin{equation}\label{eq:Corey}
f_1(U) \;=\; \frac{Au_1^2}{Au_1^2 + Bu_2^2 + Cu_3^2},\qquad
f_2(U) \;=\; \frac{Bu_2^2}{Au_1^2 + Bu_2^2 + Cu_3^2},
\end{equation}
where constants $A$, $B$ and $C$ depend on several physical quantities. It was pointed out that $u_1$ and $u_2$ are related to water and gas saturations, the oil saturation is related to $u_3 = 1 - u_1 - u_2$. The flow functions \eqref{eq:Corey} are related to the flux function for water and gas, which came from the so-called quadratic permeability Corey model; an extra implicit flow function for oil
$f_3(U) = {Cu_3^2}/({Au_1^2 + Bu_2^2 + Cu_3^2})$
is sometimes useful.
The domain is the saturation triangle given by the constraints $0 \leq u_1,\, u_2,\, u_3$ and $u_1 + u_2 + u_3 = 1$.

\subsection{Example 1: critical solution along the separatrix}
\label{sec:ex1}
A classical problem in EOR is the Water alternating Gas (WAG) injection, which has a direct relation to the Riemann problem RP$(U^L,\,O)$ where the left Riemann datum $U^L$ represents a mixture of water and gas, and the left Riemann datum $O = (0,\,0)$ represents a virgin reservoir state containing solely oil.

\begin{figure}[ht]
\sidecaption
\resizebox{7.5cm}{!}{\includegraphics{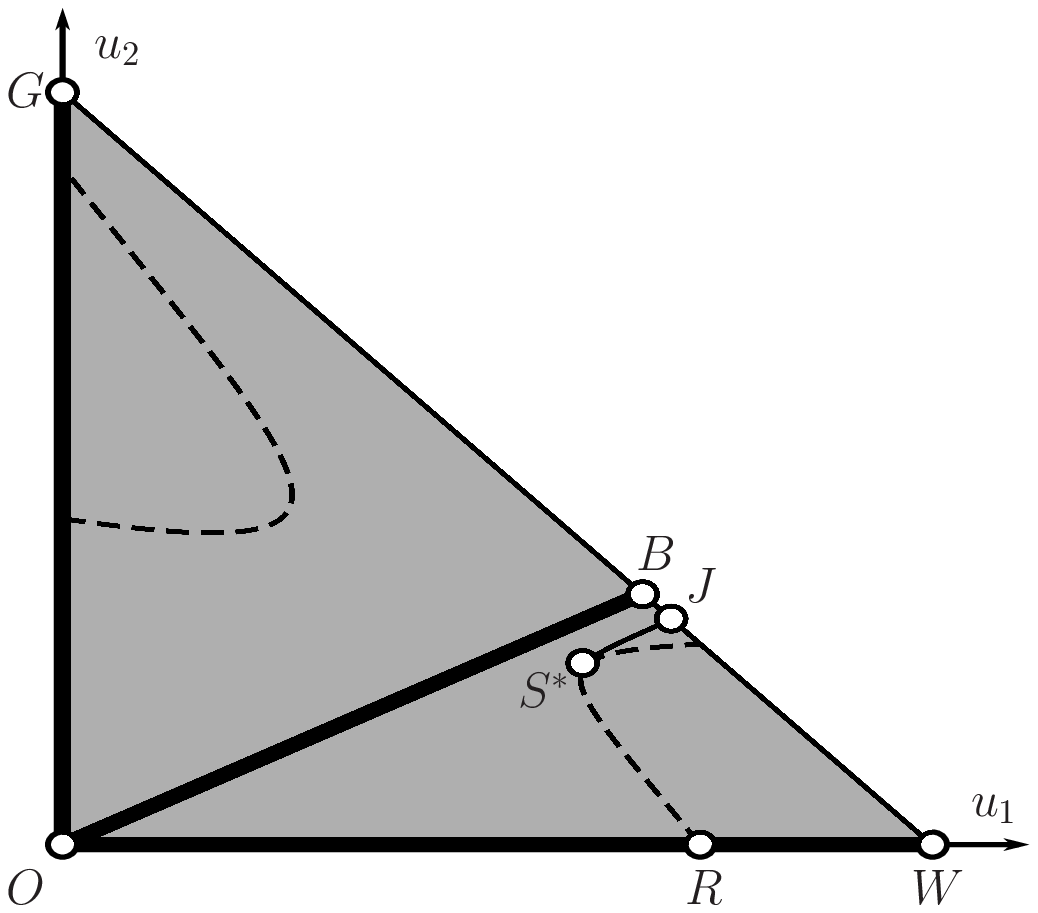}}
\caption{The shadowed region represents the saturation triangle. The vertex $O$, $W$ and $G$ are related to pure oil, water and gas, respectively. The base curves $\Gamma_i$ are the lines connecting $O$ to $W$, $G$ and $B$ respectively for $i = 1,\,2$ and $3$. The dashed curves are $\mathcal{H}(R)$, both branches of a hyperbola. The thin continuous curve is a slow-family rarefaction connecting $J$ to $S^*$.}
\label{fig:triangle}
\end{figure}

In the saturation triangle there are three base wave curves reaching $O$ ({\it cf.} \cite{Aze14}). They are lines which can be parametrized by $\ell \in I$ as the third (implicit) coordinate $u_3$ for the interval $I = [0,\,1]$. Let $\Gamma_1(\ell) = (1 - \ell,\,0)$, $\Gamma_2(\ell) = (0,\,1 - \ell)$ and, $\Gamma_3(\ell) = ((1 - \ell)B/D,\,(1 - \ell)A/D)$ be the base curves for such parametrization with the denominator $D = A + B$, see Fig.~\ref{fig:triangle}. Thus, simple computations lead to the EFFs $\f_i(\ell) = f_3(\Gamma_i(\ell))$ for $i =1,\,2,\,3$ given explicitly by
$$\f_1(\ell) =  \frac{C\ell^2}{A(1 - \ell)^2 + C\ell^2},\;\;
  \f_2(\ell) =  \frac{C\ell^2}{B(1 - \ell)^2 + C\ell^2},\;\;
  \f_3(\ell) =  \frac{C\ell^2}{AB(1 - \ell)^2/D + C\ell^2}.$$

The EFFs above satisfy the BL solution, see \cite{BL42}, with S-shaped flux functions. Their Welge points can be calculated by the values
$$\ell_1^* = 1 - \sqrt{C/(A + C)},\quad \ell_2^* = 1 - \sqrt{C/(B + C)},\quad \ell_3^* = 1 - \sqrt{CD/(AB + CD)},$$
with relative positions satisfying $\ell_3^* < \ell_1^*,\,\ell_2^*$. These inequalities guarantee that the optimal injection mixture for oil production occurs within the separatrix, see \cite{CAFM16,preprint}.

For Welge points, their values $\ell_i^*$ also indicate the locations of Bethe-Wendroff points $U^*_i := \Gamma(\ell_i^*)$. It is possible to verify that $\sigma(U^*_3,\,O) = \lambda_s(U^*_3)$ holds, so the rest of the base curve $\Gamma_3$ is a slow rarefaction. The analog Bethe-Wendroff points $U^*_1,\,U^*_2$, show that along $\Gamma_1,\,\Gamma_2$ the rarefactions are of the fast family.

\subsection{Example 2: choosing a parametrization coordinate}

The complete solution for the WAG injection needs a slow wave group. It can be found in forward direction from any point representing mixture of water and gas. However, intermediate states lie over one of the base curves of Sec.~\ref{sec:ex1}, here we show the construction of EFF over backward slow curves for states $R = (m,\,0)$.

The RH relation \eqref{rel:RH} lead to the RH locus $\mathcal{H}(R)$ for all states $U$ in the saturation triangle satisfying $f_1(U) - f_1(R) = [f_2(U)/u_2](u_1 - m)$, since the shock speed is $f_2(U)/u_2$, see (\ref{rel:RH}.b). A simple manipulation shows that $\mathcal{H}(R)$ is the edge $U = (u_1,\,0)$ and, with $f_1^R = f_1(R)$, all $U$ satisfying
\begin{eqnarray*}
(A - Af_1^R - Cf_1^R)u_1^2 &\,-\,& (B + 2Cf_1^R)u_1 u_2 \;\,-\;\, (B+C)f_1^Ru_2^2 \nonumber \\
&\,+\,& 2Cf_1^Ru_1 \;\,+\;\, (2Cf_1^R + Bm)u_2 \;\,-\;\, Cf_1^R \;\,=\;\, 0,
\end{eqnarray*}
which is a hyperbola, see also \cite{Aze14}. Of course, for the parametrization of the hyperbola in Fig.~\ref{fig:triangle}, the $u_2$ coordinate seems to be a good choice, the $u_1$ coordinate have a detour. (The third coordinate $u_3$ is also a good choice.)

The base curve $\Gamma(\ell) = (\gamma_1(\ell),\,\ell)$ is parametrized with
$$\gamma_1(\ell) = \frac{-b + \sqrt{b^2 - 4ac}}{2a},$$ 
where $a = A - (A + C)f_1^R$, $b = 2Cf_1^R - (B + 2Cf_1^R)\ell$ and $c = -(B + C)f_1^R\ell^2 + (2Cf_1^R + Bm)\ell -Cf_1^R$ hold. Thus, from the flux function (\ref{eq:Corey}.b), the EFF is
$$\f(\ell) \;=\; f_2(\Gamma(\ell)) \;=\; \frac{B\ell^2}{A\gamma_1^2(\ell) + B\ell^2 + C(1 - \ell - \gamma_1(\ell))^2},$$
at the value $\ell^*$ for the Welge point it is actually satisfied $\lambda_s(\Gamma(\ell^*)) = f_2(\Gamma(\ell^*))/\ell^*$, as proven in \cite{CF16}. (The related Bethe-Wendroff point $S^* = \Gamma(\ell^*)$ belongs to the extension boundary.) From $S^*$ the slow-family rarefaction follows until the third (implicit) coordinate vanish at $J$, see Fig.~\ref{fig:triangle}.

\section{Conclusions}
This work contributes to the applicability of EFFs for systems of conservation laws. The remarkable feature resides on restricting each wave group of a Riemann problem, into a problem with a single scalar equation. Of course, via the Wave Curve Method, the system determines the base curve which supports the EFF, therefore the Riemann problem is satisfied both for the system and for the restricted scalar conservation law.
This contribution is a first step, this engine will be important in proving conjectures we have as well as new emerging results.

\section*{Acknowledgments}
\noindent
{I'm very grateful to Prof.~Dan Marchesin (IMPA) for his friendship and, the many enlightening discussions, comments and suggestions for this work.
This work was partially supported by grants CNPq 402299/2012-4, 170135/2016-0 and FAPERJ E-26/210.738/2014. I gratefully acknowledge the financial support by Asociaci\'on Mexicana de Cultura A.C. and by the Hyp2016 conference.}

\end{document}